\documentclass[english]{amsart}

\usepackage{esint}
\usepackage[svgnames]{xcolor} 
\usepackage[colorlinks,citecolor=red,pagebackref,hypertexnames=false,breaklinks]{hyperref}
\usepackage{pgf,tikz}
\usepackage{pdfsync}

\usepackage[bottom]{footmisc}
\usepackage{dsfont}
\usepackage{url}
\usepackage[utf8]{inputenc}
\usepackage[T1]{fontenc}
\usepackage{lmodern}
\usepackage{babel}
\usepackage{mathtools}  
\usepackage{amssymb}
\usepackage{lipsum}
\usepackage{mathrsfs}
\usepackage{color}

\usepackage{epsf,graphicx,epsfig,cite,latexsym}
\usepackage[export]{adjustbox}
\usepackage{mathtools}
\usepackage{latexsym, amsmath, amssymb, a4, epsfig}
\usepackage[T1]{fontenc}

\usepackage{listings}

\usepackage{algorithm}
\usepackage{algpseudocode}

 \usepackage{mathrsfs}
\newcommand{\no}[1]{#1}
\renewcommand{\no}[1]{}
\no{\usepackage{times}\usepackage[subscriptcorrection, slantedGreek, nofontinfo]{mtpro}
\renewcommand{\Delta}{\upDelta}}
\usepackage{color}

\usepackage{tikz}
\usetikzlibrary{calc,trees,positioning,arrows,chains,shapes.geometric,decorations.pathreplacing,decorations.pathmorphing,shapes,matrix,shapes.symbols}

\tikzset{
>=stealth',
  punktchain/.style={
    rectangle, 
    rounded corners, 
    draw=black, very thick,
    text width=10em, 
    minimum height=3em, 
    text centered, 
    on chain},
  line/.style={draw, thick, <-},
  element/.style={
    tape,
    top color=white,
    bottom color=blue!50!black!60!,
    minimum width=8em,
    draw=blue!40!black!90, very thick,
    text width=10em, 
    minimum height=3.5em, 
    text centered, 
    on chain},
  every join/.style={->, thick,shorten >=1pt},
  decoration={brace},
  tuborg/.style={decorate},
  tubnode/.style={midway, right=2pt},
}

\numberwithin{algorithm}{section}
\numberwithin{figure}{section}

\newtheorem{lemma}{Lemma}[section]

\newtheorem{definition}{Definition}[section]
\newtheorem{corollary}{Corollary}[section]
\newtheorem{proposition}{Proposition}[section]
\newtheorem{theorem}{Theorem}[section]

\def\R{{\mathbf R}}
\def\M{\mathbf{M}}
\def\C{\mathbf{C}}

\def\N{\mathbf{N}}

\def\la{\langle}
\def\ra{\rangle}
\newcommand{\be}{\begin{equation}}
\newcommand{\ee}{\end{equation}}
\newcommand{\ba}{\begin{array}}
\newcommand{\ea}{\end{array}}

\newcommand{\p}{\partial}
\newcommand{\bea}{\begin{eqnarray*}}
\newcommand{\eea}{\end{eqnarray*}}
\newcommand{\bean}{\begin{eqnarray}}
\newcommand{\eean}{\end{eqnarray}}

\def\bar{\overline}
\def\hat{\widehat}
\def\tilde{\widetilde}

\def\cydot{\leavevmode\raise.4ex\hbox{.}}

\date{\today}
\title[]{ Reconstruction of domains with algebraic boundaries from generalized polarization tensors }

\author{Habib Ammari }

\address{Habib Ammari, Department of Mathematics, ETH Z\"urich, 
R\"amistrasse 101, 8092 Z\"urich, Switzerland}
\email{habib.ammari@math.ethz.ch}

\author{Mihai Putinar}

\address{Mihai Putinar, Department of Mathematics, University of California at Santa Barbara, Santa Barbara, CA 93106-3080, USA,  and School of Mathematics \& Statistics, Newcastle University Newcastle upon Tyne, NE1 7RU, 
United Kingdom}

\email{mputinar@math.ucsb.edu}

\author{Andries Steenkamp} 
\address{Andries Steenkamp, Department of Mathematics, ETH Z\"urich, 
R\"amistrasse 101, 8092 Z\"urich, Switzerland}

\email{andriess@student.ethz.ch}

\author{Faouzi Triki}

\address{Faouzi Triki,  Laboratoire Jean Kuntzmann,  UMR CNRS 5224, 
Universit\'e  Grenoble-Alpes, 700 Avenue Centrale,
38401 Saint-Martin-d'H\`eres, France}

\email{faouzi.triki@univ-grenoble-alpes.fr}

\thanks{
This work was supported in part by the
 grant ANR-17-CE40-0029 of the French National Research Agency ANR (project MultiOnde),
 the LabEx PERSYVAL-Lab (ANR-11-LABX- 0025-01), and the Swiss National Science Foundation grant number
200021--172483.}

\date{\today}
\begin{document}

\lstset{language=Matlab,%
    breaklines=true,%
    morekeywords={matlab2tikz},
    keywordstyle=\color{blue},%
    morekeywords=[2]{1}, keywordstyle=[2]{\color{black}},
    identifierstyle=\color{black},%
    stringstyle=\color{mylilas},
    commentstyle=\color{mygreen},%
    showstringspaces=false,
    numbers=left,%
    numberstyle={\tiny \color{black}},
    numbersep=9pt, 
    emph=[1]{for,end,break},emphstyle=[1]\color{red}, 
}

\begin{abstract}
 This paper aims at showing the stability of the recovery of a smooth planar domain with a real algebraic boundary from a finite number of its generalized polarization tensors.  It is a follow-up of the work [H. Ammari et al., Math. Annalen, 2018], where it is proved that the minimal polynomial with real coefficients vanishing on the boundary can be identified as the generator of a one dimensional kernel of a matrix whose entries are obtained from a finite number of generalized
polarization tensors. The recovery procedure is implemented without any assumption on the regularity of the domain to be reconstructed and its performance and limitations are illustrated. 
\end{abstract}

\maketitle
\tableofcontents 

\section{Introduction}\label{sec0}

Let $D$ be a bounded connected  Lipschitz domain in $\R^2$, and assume that its boundary 
 $\partial D$ contains the origin.     Let $\Upsilon$ be  the conductivity distribution
   in $\R^2$ given by
\bea
\Upsilon=k \chi(D) + \chi(\R^2 \setminus \overline{D}),
 \eea
where $\chi$ denotes the indicator function, and $ k$ is a fixed constant in $(0, 1)\cup (1, +\infty)$.
Let $u_0(x) = \frac{1}{2\pi} \ln|x| $,  be the fundamental solution to the Laplacian in $\R^2$.
 For a given position $z$ in $\R^2$,  we consider the following conductivity equation
\bea \label{trans}
\left\{
\begin{array}{ll}
\nabla \cdot \Upsilon \nabla u(x, z) = \delta_z(x) \quad &\mbox{in } \R^2, \\
u(x, z)-u_0(x, z) = O(|x|^{-1}) \quad &\mbox{as } |x| \to \infty, 
\end{array}
\right.
\eea
where $\delta_z$ is the Dirac function at $z$ and  $u_0(x, z) := u_0(x-z)$.
The system \eqref{trans}  has a unique solution $u$ which is the total voltage potential generated 
by the point source placed at $z$~\cite{AmmariKangb1}. The function $-\nabla u_0(x, z)$ represents the background electric field while $u(x, z)-u_0(x,z)$ is the perturbation of the voltage potential due to presence of the inclusion $D$.
Then, the far-field perturbation of the voltage potential due to the presence of $D$ is given by  \cite{AmmariKangb1}
\bean \label{asym}
u(x, z)- u_0(x, z) = \sum_{|\alpha|, |\beta| =1}^{\infty}\frac{(-1)^{|\alpha| + |\beta|}}{\alpha!\beta!}\p^\alpha u_0(x)
\M_{\alpha\beta}\p^\beta u_0(z) \qquad \textrm{as   }  |x| \rightarrow +\infty,
\eean

where and throughout this paper, we use the conventional notation:
\[x^\alpha= x_1^{\alpha_1}
x_2^{\alpha_2},\; \; \alpha= (\alpha_1, \alpha_2) \in \N^2, \; \textrm{ and}\; |\alpha| = \alpha_1+\alpha_2.\] 
We also use the graded lexicographic order: $\alpha,\, \beta \in \N^2$ verifies
$\alpha \leq \beta$ if $|\alpha| < |\beta|$, or, if $|\alpha| = |\beta|$, then $\alpha_1\leq \beta_1$  or $\alpha_1=\beta_1$ and  $\alpha_2\leq \beta_2.$

The quantities $\M_{\alpha\beta}$ that appear naturally in the multi-polar asymptotic expansion~~\eqref{asym}, are called Generalized Polarization Tensors (GPTs). We emphasize that GPTs
are not dependent on the positions $x$ and $z$. In fact  they only depend  on the inclusion $D$ and the conductivity ratio 
$1/k$ or conductivity contrast $\lambda:= \frac{k+1}{2(k-1)}$. For a fixed contrast $\lambda$, the GPTs are indeed geometric quantities associated with the shape of the domain $D$ such as eigenvalues, capacities, and moments.  The notion of GPTs has been used in diverse fields of academic research as well as of engineering applications such as the theories of composites, inverse problems,  bio-medical imaging,  bio-sensing, nano-sensing, and electro-sensing \cite{
cloaking, AKLZ, jmpa, plasmonic1, plasmonic2, touibi, BDT, TrikiVauthrin}.

From the asymptotic expansion \eqref{asym},  we deduce that the knowledge of all the GPTs
  is equivalent to knowing the far-field responses of the inclusion for all harmonic excitations.
It is well known that in that case  the inverse problem of recovering  $(\lambda, D)$ 
has a unique solution~\cite{ak03}, and a number of algorithms have been proposed for its numerical treatment~\cite{AmmariKangb1, AmmariBoulierGarnierJingKangWang, AGKLY, ak11}. However, in applications, the GPTs are usually only measured with finite accuracy and
only a  finite number of them can be determined from noisy data. Hence, studying the well-posedness of the inverse problem when only    a finite number of GPTs are available is  of importance.

The purpose of this paper is to evaluate how much information one can get from the knowledge of a finite number of these GPTs. Precisely, assuming that the domain has  an algebraic boundary,  we are interested in the inverse problem of recovering 
its position, its shape and the contrast  for given a finite number of its GPTs.  Recently the uniqueness to this inverse problem was established by the same authors~\cite{paper1}. Our goal in the present paper is twofold: (i) to  quantify the stability of the inversion and (ii) to implement  the inversion procedure and apply it to much more general cases than those discussed in \cite{paper1}. In particular, we show here how to recover the true domain (with possibly nonsmooth boundary)  from the recovered polynomial level set even in the case where several candidate domains have the same polynomial level set. In doing so, we resolve key numerical issues which include handling of bifurcation points, segmentation points, and arc sets.  It is worth emphasizing that the stability estimates proved in this paper holds for algebraic domains with smooth boundaries. Their generalization to the nonsmooth case is technically quite challenging. 

The paper is organized as follows. In Section 2, we introduce the class under consideration of domains with algebraic boundaries.  Stability issues are studied in Section 3. The main stability estimates are given in Theorem~\ref{mainStability}.  Section 4 is devoted to the presentation of our new numerical algorithm which is designed to recover algebraic domains from finite numbers of their associated GPTs. 
It is worth mentioning that based on the density with respect to Hausdorff distance of algebraic domains among all bounded domains, the proposed algorithm can be extended via approximation  beyond its natural context. This observation has already turned algebraic curves into an efficient tool for describing shapes and reconstructing them from their associated moments \cite{FatemiAminiVetterli,GustafssonMilanfarPutinar, GP, Hu, lp, TaubinCukiermanSulliven}.

\section{Real algebraic domains}\label{sec3}
In this section, we  introduce  the class of bounded open subsets in ${\bf R}^2$ with real algebraic  boundaries.  We recall the following definition.

\begin{definition} \label{alg} An open set $G$ in ${\mathbf R}^2$ is called real algebraic (or simply algebraic) if there exists a finite number of real coefficient polynomials 
$g_i(x), i=1,\cdots, m$, such that 
\bea
\p G \subset V:=\{x\in \R^2\;: g_1(x)= \cdots = g_m(x) =0\}.
\eea
 \end{definition}
 
The ellipse is a simple example of an algebraic domain, since its general boundary coincides with
the zero set of the quadratic polynomial function \[g(x)=  \sum_{|\alpha| \leq 2}g_\alpha x^\alpha \]
for given real coefficients $(g_\alpha)_{{|\alpha| \leq 2}}$ and proper signs in the top degree part.

We further denote by $\mathcal G$ the collection of bounded algebraic domains. It is well-known that the differential structure of the boundary $\p G$  consists of algebraic arcs joining finitely many singular points, see for instance \cite{delaPuente}.
 
As mentioned in \cite{paper1}, since the connectedness of the respective sets is not accessible by the linear algebra tools we developed for reconstructing an algebraic domain from a finite number of its generalized polarization tensors, we drop such a constraint here. Nevertheless, we call "domains" all elements $G \in \mathcal G$.
 
Following \cite{lp} we consider a particular class of algebraic domains which are better adapted to the uniqueness  and stability of our inverse shape problem. Let
 
\bean \label{set}
\mathcal G^* := \left \{G \in \mathcal  G:  G = \textrm{int}\, \overline G  \right \}.
\eean
 An element of $\mathcal G^\ast$ is called an {\it admissible domain}, although it may not be connected.
 
The assumption that  $G = \textrm{int}\, \overline G$ implies that $G$ contains no slits
or $\p G$ does not have isolated points.
If $G\in \mathcal G^*$, the algebraic dimension of $\p G$ is one, and the ideal associated to it is 
principal. To be more precise, $\p G$ is contained in a finite union of irreducible algebraic sets $X_j, \ j \in J,$ of dimension one each.
The reduced ideal associated to every $X_j$ is principal:
$$ I(X_j) = (P_j), \ \ \ j \in J ;$$
see, for instance, \cite[Theorem 4.5.1]{bcr}. We assume that each $P_j$ is indefinite, i.e., it changes sign when crossing $X_j$.
Therefore, one can consider the 
polynomial $g = \prod_{j \in J} P_j$, vanishing of the first-order on  $\p G$, that is
$|\nabla g| \not= 0$ on the regular locus of $\p G$. According to the real version of Study's lemma (cf. \cite[Theorem 12]
{delaPuente}) every polynomial vanishing on $\partial G$ is a multiple of $g$, that is 
$I(\p G) = (g)$. We define the degree of $\p G$ as the degree of 
the generator $g$ of the ideal $I(\p G)$. For a thorough discussion of the reduced ideal of a real algebraic surface
in ${\mathbf R}^d$, we refer the reader to \cite{DuboisEfroymson}.

Throughout this paper, we denote by $g(x)$ the single polynomial vanishing  on $\p G$ which is the
generator of $I(\p G)$ and satisfying the following normalization condition 
$g_{\alpha^*} =1$, where $\alpha^*= \max_{g_\alpha\not=0} \alpha$. We further assume that  $G\in \mathcal G^*$.

\section{Uniqueness and stability estimates}\label{sec2}
In this section,  we first recall the  uniqueness result obtained in~\cite{paper1} and  then derive 
stability estimates for the inversion procedure for smooth algebraic domains.

 \subsection{Uniqueness}
 
Let $\R[x]$ be the ring of polynomials in the variables $x = (x_1, x_2)$ and let $\R_n[x]$
be the vector space of polynomials of degree at most $n$ (whose dimension is 
 $r_n=(n+1)(n+2)/2$). Any polynomial function $p(x) \in \R_n[x]$ has a unique expansion
in the canonical basis $x^\alpha, |\alpha| \leq n$ of $\R_n[x]$, that is, 
\bea
p(x) = \sum_{ |\alpha| \leq n} p_\alpha x^\alpha, 
\eea
for some vector coefficients ${\bf p} = (p_\alpha) \in \R_{r_n}$. The following  results are established in   \cite{paper1}.
\begin{theorem} \label{main} 
Let $G \in \mathcal G^*$  with $\p G$  Lipschitz of degree $d$,  and let  $g(x)=  \sum_{ |\alpha| \leq d}g_\alpha x^\alpha$, be a polynomial function  that vanishes  of the first-order  on $\p G$,  satisfying $I(\p G) = (g),$
  $g_{\alpha^*} =1$, and $g(0)= 0$, where $\alpha^*= \max_{g_\alpha\not=0} \alpha$. Then,  there exists a discrete set $\Sigma \subset \C_0:=\C\setminus[-1/2, 1,2]$, such that for any fixed  $\lambda \in \C_0 \setminus \Sigma$,  ${\bf g} = (g_\alpha) \in \R_{r_d}$ 
is the unique solution to the following normalized linear system:
\bean \label{gg1}
{\bf p} = (p_\alpha) \in \R_{r_d}; \;\;\sum_{|\beta| \leq d}\M_{\alpha\beta}(\lambda, G)p_\beta = 0 \;\;  \textrm{for  } |\alpha| \leq 2d;\;\;
p_{\alpha^*} =1, \;\;\alpha^* =\max_{p_\alpha\not=0} \alpha. 
\eean

\end{theorem}

\begin{corollary} \label{main2}
Let $G,  \widetilde G \in \mathcal G^*$ be  Lipschitz of degree $d$. Let  $g$ and $\tilde g$ be two polynomials   
 that vanish respectively  of the first order on $\p G$  and on $\p \widetilde G$ 
satisfying $I(\p G) = (g)$ \, and\, $I(\p \widetilde G) = (\tilde g)$.  Assume that
$ g(0)= \tilde g(0) = 0$ and $\|\nabla g\|, \,\|\nabla \tilde g\|>0$ on respectively  $\p G$ and $\p \widetilde G$. Moreover, assume that $G$ is the unique element of $\mathcal G^*$ containing $0$ such that $\partial G \subset \{ g=0\} \cup B_r(0)$, where $B_r(0)$ is the disk of center $0$ and radius $r$ large enough. 
Let $\lambda   $ and $\tilde \lambda $  be fixed  in $\mathbb C_0$ such that $\lambda \notin \Sigma,\, \tilde \lambda \notin \widetilde \Sigma$, where the sets $\Sigma(\partial G)$ and  $ \widetilde  \Sigma= \Sigma( \partial \widetilde G)$ are   as defined in Theorem~\ref{main}.
 Then,  the following uniqueness result holds:
\bean \label{uniqueness}
(\M_{\alpha\beta}(G, \lambda))_{|\alpha| \leq 2d, 0<|\beta| \leq d} = (\M_{\alpha \beta}(\widetilde G, 
\tilde \lambda))_{|\alpha| \leq 2d, 0<|\beta| \leq d}
\;\;\;\; \textrm{iff}\;\;\;\; G= \widetilde G  \;\mbox{ and } \; \lambda = \tilde \lambda.
\eean
 
\end{corollary}

\begin{proof}
The result is a direct consequence of Theorem~\ref{main}. Since the generalized polarization tensors coincide, 
and $\lambda \notin \Sigma $, we can deduce from Theorem~\ref{main} that $g= \tilde g$. The fact 
that   $ g(0)= \tilde g(0) = 0$  and $\|\nabla g\|, \,\|\nabla \tilde g\|>0$ on respectively  $\p G$  and $\p \widetilde G$
implies that $G=\widetilde G$. A straightforward 
calculation shows then  that $\lambda = \tilde \lambda$, which finishes the proof.
\end{proof}


 \subsection{Stability estimates} \label{ssec3}
In this section we derive, under some regularity assumption,  stability estimates for the considered  inverse problem. For  fixed integer $d>0$, and constants $R>0$, $M_0>0$, $\kappa>0$, define a reduced set 
of algebraic domains $\mathcal G_0^*$ by

\bean \label{setG0}
\mathcal G_0^*:= \left \{ G\in \mathcal G_0^* :  G \subset B_R(0),\;  I(\p G) = (g),\; g(0)=0,\; \textrm{deg}(g)= d,\; \|{\bf g}\| \leq M_0,\; \min_{\p G} \|\nabla g\| \geq \kappa \right\}, 
\eean
where $\textrm{deg}$ denotes the degree. 
 It is not difficult to show that there exists a constant $M>M_0$, that only depends on $\mathcal G^*_0$, such that 

\bean \label{M}
|g|, \;  \|\nabla g\|,\; \|H(g)\| \leq M  \quad \textrm{on} \;  B_R(0)
\eean 

for all $g$ satisfying $I(\p G) = (g)$,  where $G \in \mathcal G_0^*$ and $H(g)$ is the Hessian matrix of $g$.\\

Let $K_1$ and $K_2$ be two compact  sets in $\R^2$. Recall that the  Hausdorff distance between $K_1$ and $K_2$ is defined by
\bea
{\bf d_H} (K_1, K_2) = \max \left\{ \sup_{x\in K_1} {\bf d}(x, K_2),\,  \sup_{x\in K_2} {\bf d}(x, K_1)\right\},
\eea
where ${\bf d}(x, K_i)=\inf_{y\in K_i} \|x-y\|, \, i=1,2$. Let $\| \; \|$ denote the Euclidean norm of tensors.

\begin{theorem} \label{mainStability} 
Let $G \in \mathcal G^*_0,  \widetilde G \in \mathcal G^*_0$  with respectively  $\p G$ and $\p \widetilde G$.
Let $\delta>0$ be a fixed  constant and   $\lambda_0 \in \R$ satisfying  $ B_\delta(\lambda_0) \Subset \C \cap \left\{ |\lambda| > \frac{1}{2} \right\} $.  Then there exists $\lambda^* \in (\lambda_0-\delta, \lambda_0+\delta)$, constants $\eta= \eta (\lambda_0, \delta, \mathcal G^*_0) \in (0,1)$, and $C= C(\lambda_0, \delta, \mathcal G^*_0)>0$, such that   if 
\[
\sum_{{|\alpha| \leq 2d, 0<|\beta| \leq d}}\left \| \M_{\alpha\beta}(\lambda^*, G) -\M_{\alpha\beta}( \lambda^*, \widetilde G)
\right \|^2 = \varepsilon^2 <1,
\]
then the following stability  result holds:
\bean \label{stab}
{\bf d_H}(\partial G, \p \widetilde G)\leq C \varepsilon^\eta. 
\eean

\end{theorem}

In order to prove Theorem \ref{mainStability}, we need to show several intermediate results.  Let  $g(x)=  
\sum_{ |\alpha| \leq d}g_\alpha x^\alpha $ and $ \tilde g(x)=  \sum_{ |\alpha| \leq d}\tilde g_\alpha x^\alpha$ be  
respectively  polynomial functions  that vanish respectively  of the first-order  on $\p G$ and $\p \widetilde G$ 
satisfying $I(\p G) = (g),\; g_{\alpha^*} =1$, \;$g(0)=0$,\, and\, $I(\p \widetilde G) = (\tilde g),\; \tilde g_{\alpha^*} =1,$
\; $\tilde g(0) = 0$. 

Further,  we shall use standard notation concerning Sobolev spaces. For a density 
$\phi \in H^{-1/2}(\partial G)$,  define the  Neumann-Poincar\'e operator:  
$\mathcal{K}^*_{G}: \, H^{-1/2}(\partial G) \rightarrow H^{-1/2}(\partial G), $
by
\bea \label{introkd2}
\mathcal{K}^*_{G} [\phi] (x) = \frac{1}{2\pi} \mbox{p.v.} \, \int_{\p G}
\frac{\la x -y, \nu_G(x) \ra}{\|x-y\|^2} \phi(y)\,d\sigma(y), \quad x \in \p G, 
\eea
where $\mbox{p.v.}$ denotes the principal value, $\nu_G(x)$ is the outward unit normal to $\partial G$ at $x \in \partial G, $ $\la\,, \,\ra$ denotes the scalar product in $\R^2$, and $\| \; \|$ denotes the Euclidean norm in $\R^2$.

 The following lemma characterizes  the resolvent set  $\rho({\mathcal K}^*_{G})$
of the operator ${\mathcal K}^*_{G}$, see, for instance, ~\cite{AmmariKangb1} and \cite{seo}.
\begin{lemma} \label{resolvent} We have
$\mathbb C\setminus 
(-1/2, 1/2] \subset \rho({\mathcal K}^*_{G})$. Moreover, if $|\lambda | 
\geq 1/2$ , then  $(\lambda I - \mathcal{K}^*_{G}) $ is invertible on 
$H_0^{-1/2}(\p G):= \{f \in H^{-1/2}(\p G): \langle f, 1\rangle_{-1/2,1/2} =0\}$. Here, $\langle \;, \; \rangle_{-1/2,1/2}$ denotes the duality pairing between $H^{-1/2}(\partial G)$ and $H^{1/2}(\partial G)$.  
\end{lemma} 

For $|\lambda|>1/2$ and  a multi-index $\alpha = (\alpha_1, \alpha_2) \in \N^2$,
 define $\phi_\alpha$ by

\bea
\phi_\alpha(y) :=(\lambda I- \mathcal{K}^*_{G})^{-1} \left [\nu_G(x)\cdot\nabla x^\alpha\right ](y), \quad 
y \in \p G.
\eea

The GPTs $\M_{\alpha\beta}$ for $\alpha, \beta\in \N^2 \; (|\alpha|, |\beta| \geq 1)$, associated with the contrast $\lambda$ and the domain $G$ can be rewritten as \cite{AmmariKangb1}  
 \bean  \label{gpt1}
 \M_{\alpha\beta}(\lambda, G) :=  \int_{\p G} y^\beta \phi_\alpha(y)\,d\sigma(y).
 \eean

Denote by $\C_\star := \C \setminus (-\infty, -2]\cup [2, +\infty),$
and let 
$\mu= \lambda^{-1} \in \mathbb C_\star$. Define respectively 
$\mathbb M(\mu) $ and $ \widetilde {\mathbb M}(\mu) $ to be the rectangular matrices with
coefficients:
\bean\label{gpt2}
 \mathbb M_{\alpha\beta}(\mu) :=  \int_{\p G}  ( I- \mu \mathcal{K}^*_{G})^{-1} \left [\nu_G(x)\cdot\nabla x^\alpha\right ] y^\beta\,d\sigma(y),\\
 \widetilde {\mathbb M}_{\alpha\beta}(\mu) :=  \int_{\p \widetilde G}  ( I- \mu \mathcal{K}^*_{\widetilde G})^{-1} \left [\nu_{\widetilde G}(x)\cdot\nabla x^\alpha\right ] y^\beta\,d\sigma(y).
 \eean

Note that $ \M_{\alpha\beta}(\lambda, G) = \lambda \mathbb M_{\alpha\beta}({1}/{\lambda})$ and $\M_{\alpha\beta}(\lambda, \widetilde G) = \lambda \widetilde{\mathbb M}_{\alpha\beta}({1}/{\lambda})$.

Recall the following result  from~\cite{paper1}. 
\begin{lemma} \label{holo} The functions
$ \mu\rightarrow \mathbb M(\mu),  \widetilde {\mathbb M}(\mu)   \in \mathcal L\left(\R^{r_d}, 
\R^{r_{2d}} \right)$  are holomorphic matrix-valued  on 
$\C_\star$. In addition,   $
\textrm{ker} (\mathbb M(0)) = \left\{c {\bf g}; \;\; c\in \R \right\}$ and   $ \textrm{ker} (\widetilde {\mathbb M}(0)) = \left\{c  \tilde{{\bf g}}; \;\; c\in \R \right\}$.
\end{lemma}

The proof of Theorem \ref{mainStability} has two main steps. In the first step,  using the  normalized linear system \eqref{gg1}, we  estimate  ${\bf g} - {\bf \tilde g}$ in terms of 
$  \mathbb M(0)-  \widetilde {\mathbb M}(0) $.  The second step consists in  applying the unique continuation of holomorphic functions  on $  \mathbb M(\mu)-  \widetilde {\mathbb M}(\mu) $  to "propagate the information" from $0$
to $\mu= \lambda^{-1}$.\\
 
Let 
\bea
F(\mu):= \sum_{{|\alpha| \leq 2d, 0<|\beta| \leq d}}\left \| \mathbb M_{\alpha\beta}(\mu) - \widetilde { \mathbb M}_{\alpha\beta}( \mu)
\right \|^2.
\eea

We remark that $F(\mu)$ is a real positive function on $\C_\star \cap \R$.
We deduce from Lemma~\ref{holo} that  $F(\mu)$ is holomorphic on $\C_\star$  and  that $F(0) = 0 $ implies
$\mathbf g =  \tilde{\mathbf g}$. 
We next estimate how much $\mathbf g$ is close to $\tilde{\mathbf g}$ when $F(0)$ is very 
small.

\begin{proposition}  \label{maininter} Let the constants $\kappa$ and $M$ be defined by (\ref{setG0}) and (\ref{M}), respectively.  Let $\varepsilon_0 =  \frac{\kappa^5}{65M^4}$ and  $ C= 64 \frac{M^4} {\kappa^5}$. Assume that 
$F(0) \leq \varepsilon_0^2$. Then the following inequality holds:
\bean \label{s1}
 {\bf d_H}^2(\partial G, \p \widetilde G) \leq C F^{1/2}(0).
\eean
\end{proposition}

In order to prove Proposition \ref{maininter} we need the following three lemmas. 
\begin{lemma}
We have 
\bean \label{fst}
 \|g -\tilde g\|_{L^2(\p G)}^2+ \|g -\tilde g\|_{L^2( \p \widetilde G)}^2 \leq 2 \kappa^{-1} M^2 F^{1/2}(0).
 \eean
\end{lemma}

\proof
From the definition of the matrices $\mathbb M(0)$ and $\widetilde{\mathbb M}(0)$, we have 

\bean\label{gpt0} \\ \nonumber
{\bf q}^t \left(\mathbb M(0)-  \widetilde{\mathbb M}(0)\right){\bf p} = 
\int_{\p G}  \nu_G(y) \cdot\nabla q(y)  p(y) \,d\sigma(y)-\int_{\p \widetilde G}  \nu_{\widetilde G}(y) 
\cdot\nabla q(y)  p(y) \,d\sigma(y),  \quad \forall p  \in  \R_{d}[x], \, q \in  \R_{2d}[x],
 \eean
 where the superscript $t$ denotes the transpose. 
 
Since $g$ and $\tilde g$  are in $ \mathcal G^*$ defined by (\ref{set}) and they respectively generate  the ideals associated to $\partial  G$ and $\p \widetilde G$, we have    $\nu_G(x) = \frac{\nabla g(x)}{\|\nabla g(x)\|}, \, x\in \p G$ and  $\nu_{\widetilde G}(x) = \frac{\nabla \tilde g(x)}{\|\nabla \tilde g(x)\|}, \, x\in \p \widetilde G$. Then \eqref{gpt0} becomes 
\bea
{\bf q}^t \left(\mathbb M(0)-  \widetilde{\mathbb M}(0)\right){\bf p} = 
\int_{\p G}  \frac{\nabla g}{\|\nabla g\|}  \cdot\nabla q(y)  p(y) \,d\sigma(y)-\int_{\p \widetilde G}  \frac{\nabla \tilde g}{\|\nabla \tilde g\|} 
\cdot\nabla q(y)  p(y) \,d\sigma(y),  \forall p  \in  \R_{d}[x], \, q \in  \R_{2d}[x].
 \eea
 
By taking $q(x)= \tilde g(x) g(x)$, $p(x)= g(x)+\tilde g(x)$, and  considering the
fact that  $g(x)$  and $\tilde g(x)$  respectively vanish on $ \partial G$ and  on  $\partial \widetilde G$, one finds that
 
 \bea
({\bf \tilde g g })^t \left(\mathbb M(0)-  \widetilde{\mathbb M}(0)\right)({\bf g +\tilde g})  =
\int_{\p G}\|\nabla g\|  \left(g -\tilde g\right)^2  \,d\sigma +\int_{\p \widetilde G}\|\nabla  \tilde g\|  \left(g -\tilde g\right)^2   \,d\sigma,
 \eea
 which in turn implies that

 \bea
\int_{\p G}\|\nabla g\|  \left(g -\tilde g\right)^2  \,d\sigma +\int_{\p \widetilde G}\|\nabla  \tilde g\|  \left(g -\tilde g\right)^2   \,d\sigma \leq 2 \left(\|{\bf g} \|^2+ \|{\bf \tilde g} \|^2\right) F^{1/2}(0).
 \eea
 
 Hence, (\ref{fst}) holds. 
 \endproof
 
For $r>0$ small,  let $\mathscr O_r \subset \R$ being the tubular domain along $\p G$, defined by
 \[
 \mathscr O_r := \left\{ y+s \nu_G(y); \, y\in \p G, \, s \in (-r, r)\right\}.
 \]

 \begin{lemma} \label{ptd}
 Assume that $0<r \leq  \frac{\kappa }{M}$. Then 
 \bean
 |g(x)| \geq \frac{\kappa}{2} r, \quad \forall x \in \p \mathscr O_r.
 \eean
   
 \end{lemma}
 \proof
Let  $x =  y +\pm r \nu_G(y)\in  \mathscr O_r$,  for some $y \in \p G$ be fixed. 
From the regularity of $g$,  it follows that the function $ s\to g(y \pm s\nu_G(y))$ is $C^2$ and  satisfies the following Taylor expansion  of order two at zero:
\[
g(y \pm r\nu_G(y)) =  \pm \nabla g(y) \cdot \nu_G(y) r + \frac{r^2}{2}\nu_G^t(y)H(g)(y \pm s_0 \nu_G(y))  \nu_G(y),
\]
where $H(g)(y)$ is the Hessian matrix of $g$  at $y$, and $s_0 $ is some constant in between $0$ and $\pm r$. Recalling that  $\nu_G(x) = \frac{\nabla g(x)}{\|\nabla g(x)\|}$, we therefore obtain that
\bea
|g(x)| \geq    \kappa  r - M\frac{r^2}{2},
\eea
which finishes the proof.
 
 \endproof
 The proof of Lemma \ref{ptd} shows  that if the zero  level set of $g$ is isolated, that is,  $\|\nabla g\| \not = 0$ on $\p G$, then the polynomial $g$ behaves as a weighted signed distance function to  the boundary $\p G$ in the small   tubular neighborhood domain $\mathscr O_r$. 
 
 \begin{lemma} \label{p2p}Let $r^* = \frac{\kappa }{M}$ and $\varepsilon_0 =  \frac{\kappa^5}{65 M^4}$.
Assume that   $F(0) \leq  \varepsilon_0^2$. Then 
 \bean \label{pop}
 \p \widetilde G \subset \mathscr O_{r^*}.
 \eean
 \end{lemma}
 \proof
 Let $\tilde x(t)$ be the parametric representation of the boundary $\p \widetilde G$ ($\p {\widetilde G} = \left\{ \tilde x(t), \; t\in \R_+ \right\}$) satisfying 
 \bean \label{param}
 \frac{d \tilde x}{dt}(t) = J \nabla \tilde g( \tilde x(t)),  \quad t >0, \;\; \textrm{and} \;\;  \tilde x(0) = 0,
 \eean
 where $J$ is the counter-clockwise rotation  matrix by $\pi/2$. Since $\tilde g$ is smooth,  $\tilde x(t)$ is the 
 unique solution to the system~\eqref{param}, which is in addition of class  $C^1$ and is periodic on $\R_+$.
 
 Now we shall prove that $\tilde x(t)$ lies indeed in  $\mathscr O_{r^*}$, for all $t \in \R_+$. Assume  that $ \p \widetilde G$ is not entirely included in $ \mathscr O_{r^*}$, and
 define
 \[
 t_0 = \sup \{ t\in \R_+:  \;\; \tilde x(t) \in \mathscr O_{r^*}\}.
 \]
 Since $0 \in \p G$,   $t_0>0$ is well defined,  is finite, and 
verifies $\tilde x (t_0) \in \p \mathscr O_{r^*} $. Lemma~\ref{ptd}
 then implies that
 \bean \label{egt}
 |g(\tilde x(t_0))| \geq \frac{\kappa}{2}.
 \eean
 
In view of the regularity of $g$ and  since $\tilde x$ verifies~\eqref{param}, we have
\bean \label{egt2}
|g(\tilde x(t)) - g( \tilde x(s))| \leq M^2 |t-s|, \quad \forall s, t \in \R_+. 
\eean
Combining  inequalities \eqref{egt} and \eqref{egt2}, we obtain that 
\bea
|g(\tilde x(t))| \geq    \frac{\kappa}{4},
\eea
for all $t$ satisfying  $|t-t_0| \leq \frac{\kappa}{4M^2}$.
Whence
\[
\| g- \tilde g \|_{L^2(\p \widetilde G)}^2 \geq    \int_{t_0-\frac{\kappa}{4M^2}}^{t_0+\frac{\kappa}{4M^2}}
|g(\tilde x(t) )|^2 \| \nabla \tilde g(\tilde x(t))\|  dt \geq  \frac{\kappa^4}{32M^2}.
\]

This together with \eqref{fst} entail
\[
F^{1/2}(0) \geq \frac{\kappa^5}{64 M^4},
\]
which is  in contradiction with the fact that $F(0) \leq \varepsilon_0^2$. Then the inclusion  \eqref{pop} is satisfied.

  \endproof
  
\begin{proof}[Proof of Proposition~\ref{maininter}]  Now, we are ready to prove Proposition~\ref{maininter}. We further assume that $F(0) \leq \varepsilon_0^2$. 
  Let $\tilde x(t)$ be defined by  \eqref{param}.
 Since  $\p \widetilde G \subset \mathscr O_{r^*}$,  for each $t>0$, there exists $r(t) \in (0, r^*)$ and $y(t)
 \in \p G $, such that $x(t) = y(t)\pm r(t) \nu_G(y(t))$.  Noting that $x(t) \in \mathscr O_{r(t)}$, we get from
 Lemma~\ref{ptd} the following estimate:
 \bea
 |g(\tilde x(t))| \geq \frac{\kappa}{2} r(t).
 \eea
Following  the same arguments as those in the proof of Lemma \ref{p2p}, we get
 \bea
|g(\tilde x(s))| \geq    \frac{\kappa}{4} r(t),
\eea
for all $s$ satisfying  $|t-s| \leq \frac{\kappa}{4M^2}$.
Whence
\[
\| g- \tilde g \|_{L^2(\p \widetilde G)}^2 \geq    \int_{t_0-\frac{\kappa}{4M^2}}^{t_0+\frac{\kappa}{4M^2}}
|g(\tilde x(t) )|^2 \| \nabla \tilde g(\tilde x(t))\|  dt \geq  \frac{\kappa^4}{32M^2} r^2(t),
\]
for all $t\in \R_+$. Then 
\[
\frac{\kappa^4}{32 M^2}{\bf d}^2(\tilde x(t), \p G) \leq  \frac{\kappa^4}{32M^2} r^2(t) \leq \| g- \tilde g \|_{L^2(\p \widetilde G)}^2,
\]
 for all $t\in \R_+$, which implies  
 \bean \label{dd1}
 \frac{\kappa^4}{32M^2} \sup_{x\in \p \widetilde G} {\bf d}^2(x, \p G)\leq  \| g- \tilde g \|_{L^2(\p \widetilde G)}^2.
 \eean
 Repeating the same steps by interchanging  $G$ and  $\widetilde G$, we also get 
 \bean \label{dd2}
 \frac{\kappa^4}{32M^2} \sup_{x\in \p  G} {\bf d}^2(x, \p \widetilde G) \leq  \| g- \tilde g \|_{L^2(\p G)}^2.
 \eean
Finally, combining   inequalities \eqref{dd1},  \eqref{dd2}, and  \eqref{fst}, we obtain the final result of 
Proposition \ref{maininter}. 
 \end{proof}

The second step in proving Theorem \ref{mainStability} consists in showing the following proposition. 
\begin{proposition}  
Let $\delta>0$ be a fixed  constant and   $\lambda_0 \in \R$ satisfying  $ B_\delta(\lambda_0) \Subset \C \cap \left\{ |\lambda| > \frac{1}{2} \right\} $.  Then, there exist
 constants $\theta = \theta(\lambda_0, \delta)>0$ and  $C= C(\lambda_0, \delta)>0$ such that   

\bean \label{s2}
F(0) \leq C \left \| \lambda \mapsto F(\frac{1}{\lambda})\right\|_{L^\infty ((\lambda_0-\delta, \lambda_0+\delta))}^\theta.
\eean
\end{proposition}
\proof Let   $\omega \Subset B_2(0)$  be the image of $(\lambda_0-\delta, \lambda_0+\delta)$ by the complex function $\lambda \mapsto 1/\lambda$. Then there
exists a constant  $r_0 \in (0, 2)$ such that $\omega \Subset B_{r_0}(0)$. Denote by $M_1 =\|\mu \mapsto F(\mu)\|_{L^\infty(B_{r_0}(0))}$,
and  let $w$ be the harmonic measure satisfying
$$\left\{
\begin{array}{lll}
\Delta w&=& 0 \quad \textrm{in}\; B_{r_0}(0)\setminus \overline \omega,\\
w &=& 0 \quad \textrm{on} \; \p B_{r_0}(0), \\
 w &=& 1 \quad \textrm{on}\;  \p \omega.
 \end{array} \right. $$
Since $\mu \mapsto F(\mu)$ is holomorphic on $B_{r_0}(0)$, the function $\mu \mapsto \log|F(\mu)|$ is subharmonic, and we can deduce from  the  Two constants Theorem~\cite{Nevanilinna} the following inequality:
\bea
F(\mu) \leq M_1^{1-w(\mu) }\|\mu \mapsto F(\mu)\|_{L^\infty(\omega)}^{w(\mu)}.  
\eea
Then by taking $\theta = w(0)$, and $C= M_1^{1-w(0)}$, we obtain the result. 
\endproof

\begin{proof}[Proof of Theorem \ref{mainStability}] Finally, we are now in a position to prove Theorem \ref{mainStability}.  
Let $\lambda^* \in  (\lambda_0-\delta, \lambda_0+\delta)$.
By combining estimates \eqref{s1} and  \eqref{s2} together with  the fact that $\M_{\alpha\beta}= \lambda \mathbb M_{\alpha \beta}$, we finally obtain the desired stability result stated in Theorem \ref{mainStability}.
\end{proof}

\section{Algorithm description and numerical examples}
\subsection{Algorithm}
Before we can dive into the algorithm for recovering algebraic domains from finitely many of their GPTs we must first define a processed form of the GPTs that will form our starting point. In \cite[Algorithm 6.2]{paper1} the GPTs $(\textbf{M}_{\alpha \beta})_{|\alpha|\leq 2d,|\beta|\leq d}$ are flattened out into a linear system. We define one such system explicitly here. For doing so, we use the notation $\textbf{M}_{\alpha \beta}= \textbf{M}_{[\alpha_1,\alpha_2], [\beta_1,\beta_2]}$, where $\alpha=(\alpha_1,\alpha_2)$ and $\beta= (\beta_1,\beta_2)$. 

\begin{definition}{The GPT tessera of order (m,n)} is given by
\begin{gather*}
\tilde{\textbf{M}}_{m,n} : =
\begin{bmatrix}
\textbf{M}_{[m,0],[n,0]}(\lambda, G)&\textbf{M}_{[m,0],[n-1,1]}(\lambda, G)&\cdots&\textbf{M}_{[m,0],[1,n-1]}(\lambda, G)&\textbf{M}_{[m,0],[0,n]}(\lambda, G)\\
\textbf{M}_{[m-1,1],[n,0]}(\lambda, G)&\textbf{M}_{[m-1,1],[n-1,1]}(\lambda, G)&\cdots&\textbf{M}_{[m-1,1],[1,n-1]}(\lambda, G)&\textbf{M}_{[m-1,1],[0,n]}(\lambda, G)\\
\vdots&\vdots&\ddots&\vdots\\
\textbf{M}_{[1,m-1],[n,0]}(\lambda, G)&\textbf{M}_{[1,m-1],[n-1,1]}(\lambda, G)&\cdots&\textbf{M}_{[1,m-1],[1,n-1]}(\lambda, G)&\textbf{M}_{[1,m-1],[0,n]}(\lambda, G)\\
\textbf{M}_{[0,m],[n,0]}(\lambda, G)&\textbf{M}_{[0,m],[n-1,1]}(\lambda, G)&\cdots&\textbf{M}_{[0,m],[1,n-1]}(\lambda, G)&\textbf{M}_{[0,m],[0,n]}(\lambda, G)\\
\end{bmatrix}.
\end{gather*}
\end{definition}

\begin{definition}{The Tesselated GPT (TGPT) of order (d)} is given by
\begin{gather*}
\textbf{TGPT}_{2d,d}:=\begin{bmatrix}
\tilde{\textbf{M}}_{1,1}&\tilde{\textbf{M}}_{1,2}&\cdots&\tilde{\textbf{M}}_{1,d}\\
\tilde{\textbf{M}}_{2,1}&\tilde{\textbf{M}}_{2,2}&\cdots&\tilde{\textbf{M}}_{2,d}\\
\vdots&\vdots&\ddots&\vdots\\
\tilde{\textbf{M}}_{2d,1}&\tilde{\textbf{M}}_{2d,2}&\cdots&\tilde{\textbf{M}}_{2d,d}\\
\end{bmatrix}.
\end{gather*}
\end{definition}

Our algorithm has in total nine steps. The detail of each step is given algorithmically below with an accompanying description and diagrams. The main steps consist in first recovering the polynomial level set from the given GPTs, then then reconstructing the domain candidates and finally selecting one of the domain candidates in order to minimise the discrepancy between its GPTs and those of the true domain. Our algorithm goes far beyond the stability estimates established in the previous section. Here there is no need to assume that the curve to be recovered is smooth. Nevertheless, in order to reconstruct the domain candidates, several issues need to be carefully resolved. These include bifurcation points, segmentation points, and arc sets.  

There are also tuning parameters scattered throughout the various processes and for the most part they are fixed. These tuning parameters should not distract from the otherwise straightforward process.

\begin{figure}\scalebox{.9}{
\begin{tikzpicture}
  [node distance=.8cm,start chain=going below,]

     \node[punktchain, join] (TGPT) {TGPT};
     \node[punktchain, join, ] (Coef) {Coefficients};
     \node[punktchain, join, ] (Poly) {Polynomial};
	 \begin{scope}[start branch=hoejre,]
	 
    \node [punktchain, on chain=going right, join=by {->}] (Single loop Domian) {Single Loop Domain};
     
     \end{scope}      
     
     \node[punktchain, join, ] (SegPoints) {Segmentation points};

     \begin{scope}[start branch=venstre,
      every join/.style={->, thick, shorten <=1pt}, ]
     \node[punktchain, on chain=going left, join=by {<-}]
             (Bifur) {Bifurcation points};
     \end{scope}
     	 \draw[|-,-|,->, thick,] (Poly.south) |-+(0,-1em)-| (Bifur.north);
	 \begin{scope}[start branch=hoejre,]
     \node [punktchain, on chain=going right, join=by {->}] (ArcSet) {Arc set};
     \end{scope}     
     \draw[|-,-|,->, thick,] (Poly.south) |-+(0,-1em)-| (ArcSet.north);      
	 \node  [punktchain ] (RecDom) {Recovered Domains}; 	 
	 
	 \begin{scope}[start branch=hoejre,]
     \node [punktchain, on chain=going right, join=by {<-}] (CircSet) {Circuit set};
     \end{scope} 
	 
     \draw[|-,-|,->, thick,] (ArcSet.south) |-+(0,-1em)-| (CircSet.north); 	 
      \draw[|-,-|,->, thick,] (ArcSet.south) |-+(0,-1em)-| (RecDom.north); 
      \draw[|-,-|,->, thick,] (SegPoints.south) |-+(0,-1em)-| (RecDom.north);     
  \end{tikzpicture}}
        \caption{ This diagram shows in broad terms the process we take to recover a domain from an associated TGPT.}
        \end{figure}
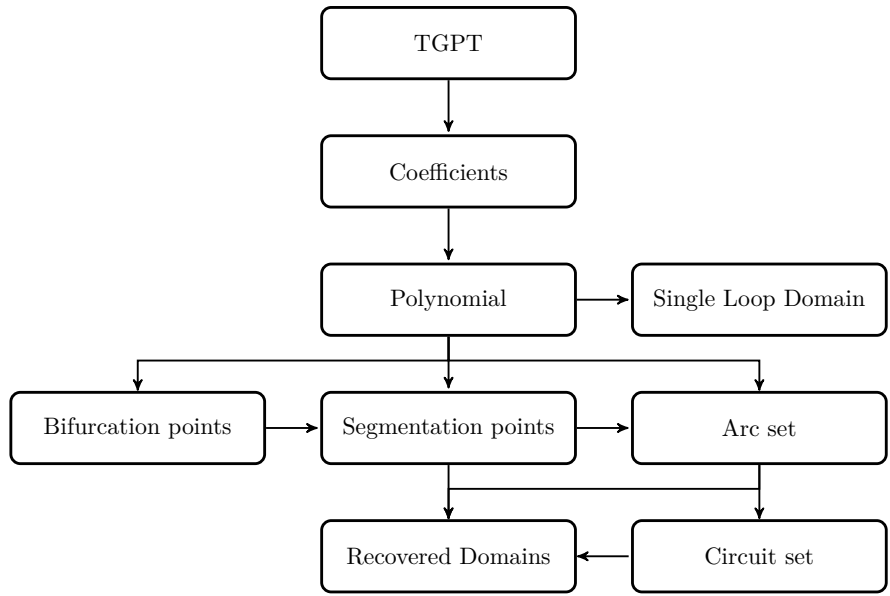

\begin{algorithm}
\caption{Recover Domain}\label{Polynomial}
\begin{algorithmic}[1]
\Procedure{recDom}{$\textbf{TGPT}_{2d,d},\lambda$}
\State $\overline{g} \gets $Algorithm 6.2.$(TGPT,\lambda)$ 
\State $P(x,y) := \Sigma_{i,j}^{n}g_{i,j}x^{i}y^{j} \gets \overline{g} $
\State $D' \gets$ \textsc{checkLoop}{$(P(x,y),t_{step},tol_{po},tol)$}
\If {$D' \neq \emptyset$}
\State \textbf{return} $D'$
\EndIf
\State $B \gets$ \textsc{GetBifurcationPoints}{$(P(x,y),a,b,tol_{bif})$}
\State $S \gets$ \textsc{GetSegmentationPoints}{$(P(x,y),B,r_{ini},r_{step},N)$}
\State $E \gets$ \textsc{FindArcs}{$(P(x,y),S,Bound)$}
\State $C \gets$ \textsc{FindCircuits}{$(E,B,S)$}
\State $\mathcal{D} \gets$ \textsc{ConstructDomains}{$(P(x,y),S,C,t_{step},tol)$}
\State $\bar{D} \gets$ \textsc{rankDomains}{$(\mathcal{D},\textbf{TGPT}_{2,1})$}
\State \textbf{return} $\bar{D}$		
\EndProcedure

\textbf{Description:} This is the wrapper that calls the individual procedures that constitute the algorithm. It is included as to see the sequence of steps. The assumed starting point of the algorithm is $\textbf{TGPT}_{2d,d}$. The TGPT is obtained from   \cite[Algorithm 6.1]{paper1}. We now go into the details of each step.
\end{algorithmic}
\end{algorithm}

\begin{algorithm}
\caption{Check for a Loop}\label{CheckLoop}
\begin{algorithmic}[1]
\Procedure{checkLoop}{$P(x,y),t_{step},tol_{po},tol$}
\State $ \{[x_p,y_p]:p \in 1,..,N \}\gets $ \textsc{TraceLvlSet}$(P(x,y), [0,0], [0,0],t_{step},tol_{po})$
\State $\overline{D} := \{[x_p,y_p]:p \in 1,..,N \}$
\If {$\|[x_0,y_0]-[x_N,y_N]\| < tol$}
\State \textbf{return} $\overline{D}$
\Else
\State \textbf{return} $\overline{D} = \emptyset$
\EndIf
\EndProcedure

\textbf{Description:}
The purpose of this step is to confirm if the recovered polynomial level set is not already a smooth Jordan curve. If this is the case, the rest of the algorithm is unnecessary and inapplicable. 
To confirm, we trace out the level set using Algorithm \ref{PolynomialOrbit} with the origin as an initial point and terminal points. Minor technicalities are involved in order to make sure that the procedure does not stop exactly where it begins.
\end{algorithmic}
\end{algorithm}

\begin{algorithm}
\caption{Polynomial Level set trace}\label{PolynomialOrbit}
\begin{algorithmic}[1]
\Procedure{TraceLvlSet}{$P(x,y),p_{0}, T, dir = 1, t_{step}, tol_{po}$}
\State $H(x,y) := [-\partial_y P(x,y) , \partial_x P(x,y)  ]$
\State $[x_{0},y_{0}] = p_0 ~;~ t_0 = 0$
\While {End-condition $=$ false}
\State $[x_{n},y_{n}] = dir\cdot H(x_{n-1},y_{n-1})t_{n-1} +  [x_{n-1},y_{n-1}]$
\State $t_n = t_{n-1} + t_{step}$
\If {$\min_{\tau \in T}||[x_{n},y_{n}]-\tau|| < tol_{po}$ }
\State End-condition $=$ true 
\EndIf
\EndWhile
\State $N := argmin_{n \in \mathbb{N}}||[X_{n},Y_{n}]-\tau||:{\tau \in T}$ 
\State \textbf{return} $\{[X_{n},Y_{n}]\}_{n \in \{0,1,...,N \}}$
\EndProcedure

\textbf{Description:}
The core notion of this procedure is the following two steps. Firstly define an equation of motion from the polynomial. Secondly use this equation to move along the level set starting from a known point on the level set. The equation of motion is given in line $5$ and uses function $H(x,y)$ which is the gradient of $P(x,y)$ rotated by $\pi/2$. $H(x,y)$ is called the Hamiltonian and is tangent to the level set for points $(x,y)$ on the level set. The tracing out is done by a Runga-Kutta algorithm. The stop condition is defined by a set $T$. The stop condition is hence that the traced level set reaches a specified proximity to a point in $T$. The set $T$ can consist of a single or several points.
\end{algorithmic}
\end{algorithm}

\begin{algorithm}
\caption{Bifurcation points}\label{BifPoints}
\begin{algorithmic}[1]
\Procedure{GetBifurcationPoints}{$P(x,y),a,b,tol_{bif}$}
\State $F(x,y) := [P(x,y), \partial_x P(x,y), \partial_y P(x,y)] $ 
\State $B_{pre} := argmin_{(x,y) \in [a,b]^2} F(x,y)$
\State $B \gets $Cluster points in $B_{pre}$ that have distance $< tol_{bif}$
\State \textbf{return} $B= \{[b_{x}^{(i)},b_{y}^{(i)}]\}_{i \in M}$
\EndProcedure

\textbf{Description: }The recovered polynomial level set consists of finitely many smooth arcs. These arcs meet at what is called bifurcation points. Bifurcation points are easily found by minimizing $P(x,y)$ and its derivatives. The order of derivatives  dependent on the number of arcs meeting. For our purposes it was sufficient to only minimize the first. Two things to note, $a,b$ specify a box within which there is searched and $tol_{bif}$ is the threshold for the minimization. The code used would automatically increase $tol_{bif}$ until at least two bifurcation points were found.   
\end{algorithmic}
\end{algorithm}

\begin{algorithm}
\caption{Segmentation points}\label{SegPoints}
\begin{algorithmic}[1]
\Procedure{GetSegmentationPoints}{$P(x,y),B,r_{ini},r_{step},N$}
\For {$b_{i} = [b_{x}^{(i)},b_{y}^{(i)}]~:~ i \in M$}
\State $\bar{s}_i = \{\}$ 
\State $r = r_{ini} $
\While {$|\bar{s}_i| < 4$}
\State $[x^{(i)}_{j},y^{(i)}_{j}] = b_{i} + [r \cos(\theta_j)~,~r \sin(\theta_j)]~\forall~ \theta_j := 2\pi \frac{j}{N};~j \in \{1,2,...,N\}$
\If {$|\{j: P(x^{(i)}_{j},y^{(i)}_{j}) = 0\}| > 3 $}
	\State $\bar{s}_i := \{[x^{(i)}_{k},y^{(i)}_{k}]:P(x^{(i)}_{k},y^{(i)}_{k}) = 0~,~k \in M_{i} \}$
\Else
\State $r = r + r_{step}$
\EndIf
\EndWhile
\State $S \gets \bar{s}_i$
\EndFor
\State \textbf{return} $S =  \{[x^{(i)}_{k},y^{(i)}_{k}] : k \in M_i ~,~i \in M\}$ 
\EndProcedure

\textbf{Description: }As seen in Figure \ref{fig1b} the bifurcation points seldom lie on the level set. We now seek the nearest points on the level set to a fixed bifurcation point. These segmentation points  define the end points of arcs in the level set. Note the double index notation, which is useful for defining arcs. The parameter $N$ determines the fineness of the minimization and was fixed at $1000$ and left at that.
\end{algorithmic}
\end{algorithm}

\begin{figure}
  \includegraphics[scale =0.6]{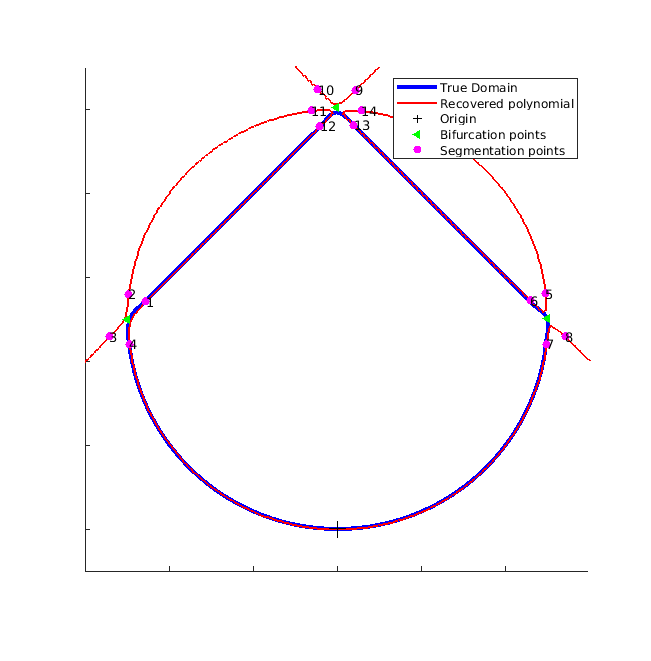}
  \caption{Recovered polynomial level set with bifurcation  and segmentation points.}\label{fig1a}
\end{figure}

\begin{figure}
  \includegraphics[scale =0.6]{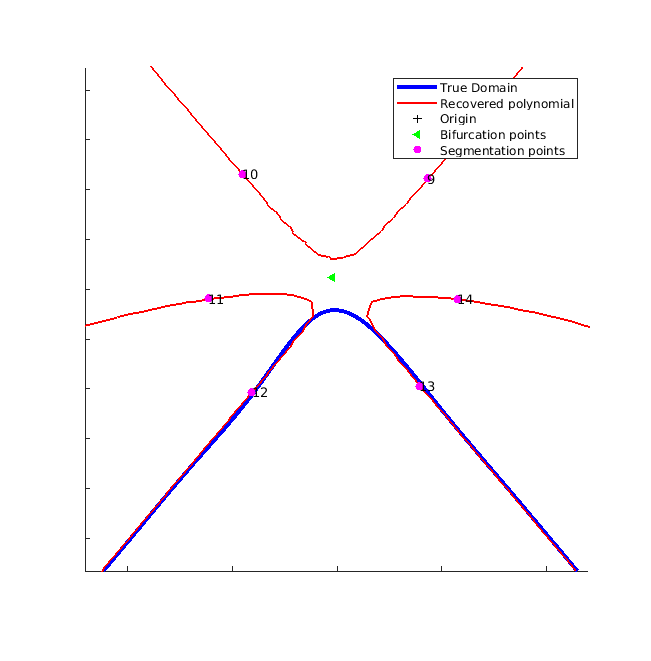}
  \caption{Close-up of the top bifurcation point.} \label{fig1b}
\end{figure}

\begin{algorithm}
\caption{Arcs}\label{Arcs}
\begin{algorithmic}[1]
\Procedure{FindArcs}{$P(x,y),S,Bound$}
\State TrivArcs $:= \{(ik,il,0):k\neq l  \wedge k,l \in M_i ~,~i \in M\}  $ 
\For{$[x^{(i)}_{k},y^{(i)}_{k}]\in S$} 
\State $\{[x_p,y_p]:p \in \{1,..,N\} \} = $\textsc{TraceLvlSet}$(P(x,y), s, S\setminus \{s\},1,t_{step})$
\State $[x^{(j)}_{l},y^{(j)}_{l}] =:\hat{s} = argmin_{s' \in S\setminus\{s\}} ||[x_N,y_N]-s'||$
\If {$e_i \in$ TrivArcs \textbf{or} $||[x_N,y_N]|| > Bound$}
\State Skip to next segmentation point.
\Else
\State $e_{ik} := (ik,jl,1)$ Positive direction
\State $e_{-ik} := (jl,ik,2)$ Negative direction
\State $E \gets \{e_{ik}, e_{-ik}\}$
\EndIf
\EndFor
\State $E \gets$ TrivArcs
\State \textbf{return} $E$
\EndProcedure

\textbf{Description: } The task of this procedure is to find which pairs of segmentation points are connected through the level set.
The connection is described as an ordered triple. The first two entries are the indices of the segmentation end points. The third entry
is the direction of motion along the level set. Zero is used when the arc does not lie on the level set, see Figure \ref{fig1b}. The parameter $Bound$ here is just to ensure arcs do note race off to infinity.
\end{algorithmic}
\end{algorithm}
\begin{figure}
  \includegraphics[scale =0.6]{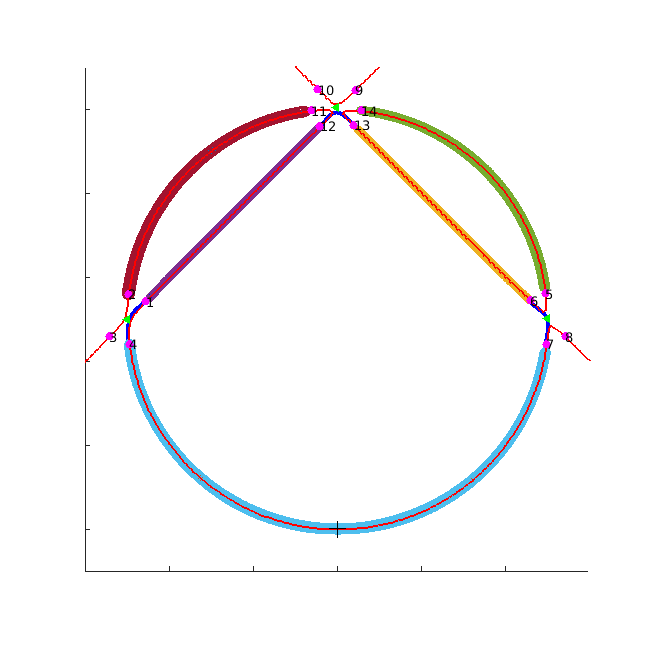}
  \caption{The various allowed arcs recovered from the level set, displayed each in a unique colour.}\label{fig1c}
\end{figure}

\begin{algorithm}
\caption{Circuits}\label{Circuits}
\begin{algorithmic}[1]
\Procedure{FindCircuits}{$E,B,S$}
\State $C' \gets$ Algorithm from \cite{hilblemniscat}
\For {$c \in C'$}
\State Remove $c$ if $|c| < 4$
\State Remove $c$ if $|c| > 2 |B|$
\State Remove $c$ if it does not containing the origin.
\State Remove $c$ if it visits the same bifurcation point twice.
\State Remove $c$ if it is a variation of a previous circuit.
\EndFor
\State The result is $C \subset C'$
\State \textbf{return} $C$
\EndProcedure

\textbf{Description: }This procedure has the goal of making circuits from the previously obtained directed arcs. For clarity this procedure is represented in a more simplified way than the others. The first step is to use a well-known algorithm like the one in  \cite{hilblemniscat} in order to find all elementary circuits. Other algorithms are also viable as the arc set is quite small. These circuits represent domain candidates. To reduce the number of candidates, we incorporate some information on the domain. This information takes the form of constraints on the size, inclusion of the origin and internal bifurcation points. 
\end{algorithmic}
\end{algorithm}

\begin{algorithm}
\caption{Constructed Candidate Domains}\label{conDomains}
\begin{algorithmic}[1]
\Procedure{ConstructDomains}{$P(x,y),S,C,t_{step},tol$}
\For {$c \in C$}
\For {$(ik,jl,dir \neq 0) :=e \in c$}
\State $s := [x^{(i)}_{k},y^{(i)}_{k}]$
\State $s' := [x^{(j)}_{l},y^{(j)}_{l}]$
\State $\{[x^{(e)}_j,y^{(e)}_j]:j \in [N_{e}] \}$ = \textsc{TraceLvlSet}$(P(x,y), s, s',dir ,t_{step})$
\EndFor
\For {$(ik,jl,dir = 0) :=e \in c$}
\State $\{[x^{(e)}_j,y^{(e)}_j]:j \in [N_{e}] \} \gets $ interpolate from other arcs.
\EndFor
\State $D^{(c)} := \bigcup_{e \in c} \{[x^{(e)}_j,y^{(e)}_j]:j \in [N_{e}] \}$
\EndFor
\State $\mathcal{D} := \{D^{(c)}~:~c\in C\}$
\State \textbf{return} $\mathcal{D}$
\EndProcedure

\textbf{Description: } This procedure is used to convert a circuit into a set of boundary points. The circuit can be thought of as a blue print for the domain candidate. This is because the circuit defines the sequence of arcs that constitute a domain. Hence the procedure traces out these arcs using the respective segmentation points as start and stop points. The gaps in between arcs obtained from tracing the level set are filled via interpolation from the existing arcs. The result is a set of domains in the form of a set of boundary points.
\end{algorithmic}
\end{algorithm}

\begin{algorithm}
\caption{Rank Domains}\label{rankDomains}
\begin{algorithmic}[1]
\Procedure{rankDomains}{$\mathcal{D},\textbf{TGPT}_{2,1}$}
\For {$D^{(c)} \in \mathcal{D}$}
\State Export $D^{(c)}$ as $image_{c}.png$ 
\State Read $image_{c}.png$ as a curve (see \url{https://github.com/yanncalec/SIES}) 
\State Compute $\textbf{TGPT}_{2,1}^{(c)}$.
\State  $\bar{c} = argmin_{c}\frac{\|\textbf{TGPT}_{2,c}^{(l)} - \textbf{TGPT}_{2,c}\|}{\|\textbf{TGPT}_{2,c}\|}$
\State $\bar{D} \gets D_{\bar{c}}$
\EndFor
\State \textbf{return} $\bar{D}$
\EndProcedure

\textbf{Description: }This final procedure is to discern which of the finite set of domain candidates most closely resembles the true domain. The resemblance is determined by the first order TGPT. The reason for the export step is that it sub-samples the domain candidate. Otherwise the recovered domain contains far too many points to be numerically stable. 
\end{algorithmic}
\end{algorithm}

\begin{figure}
  \includegraphics[scale =0.6]{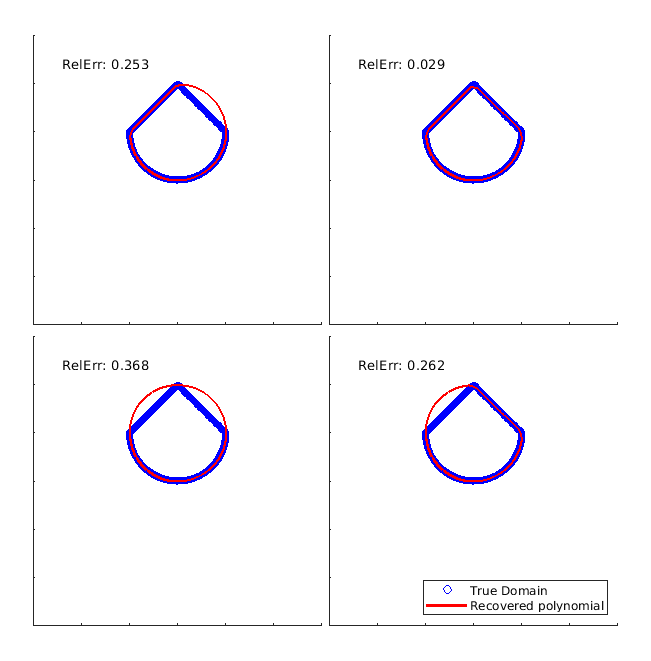}
  \caption{Recovered domains compared to the true domain.}\label{fig1d}
\end{figure}

\clearpage

\subsection{Examples}
In this section, we apply the algorithm described in the previous subsection to a few examples. We demonstrate its performance by means of a well chosen examples. We also show where the algorithm fails.

In the first example, Figure \ref{Ex1b} present the possible seven domain candidates corresponding to a disk with a sector missing shown in Figure \ref{Ex1a}. The true domain is recovered by Algorithm \ref{rankDomains}. Here, it corresponds to the one with relative error $0.021$.

\begin{figure}
  \includegraphics[scale =0.6]{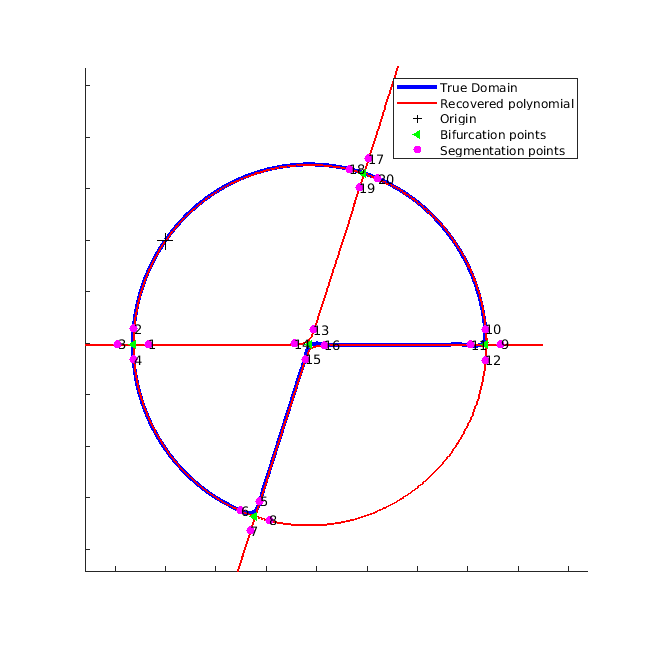}
  \caption{Figure of a disk with a sector missing.}\label{Ex1a}
\end{figure}

\begin{figure}
  \includegraphics[scale =0.6]{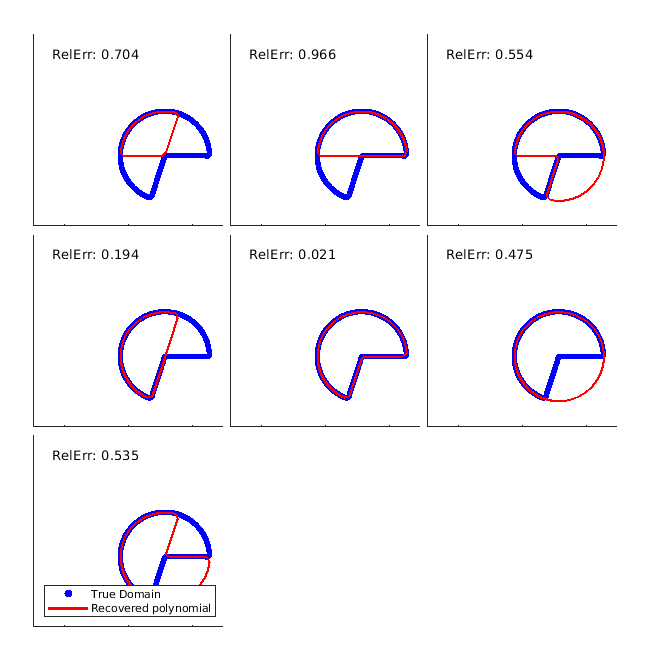}
  \caption{Figure of viable domain candidates.}\label{Ex1b}
\end{figure}
\clearpage

In the second example, we consider domains with the same recovered level set. These are discerned from each other  by using some boundary information and matching  the associated TGPTs. Figures \ref{Ex2b}, \ref{Ex2c}, and \ref{Ex2d} show 
three of the six distinct domains.  We call these domains "conjoined circles", "crescent"  and "intersection of circles" respectively to indicate the shape.  All of these shape have the same level set namely two overlapping circles as seen in
 Figure \ref{Ex2a}. Among the candidates of the conjoined circles the best candidate was found to have relative error $0.01$, see Figure \ref{Ex2b}. Among the candidates of the crescent the best candidate was found to have relative error $0.053$, see Figure \ref{Ex2c}. And among the candidates of the intersection of circles shape the best candidate was found to have relative error $0.044$, see Figures \ref{Ex2d}.

\begin{figure}[H]
  \includegraphics[scale =0.6,center]{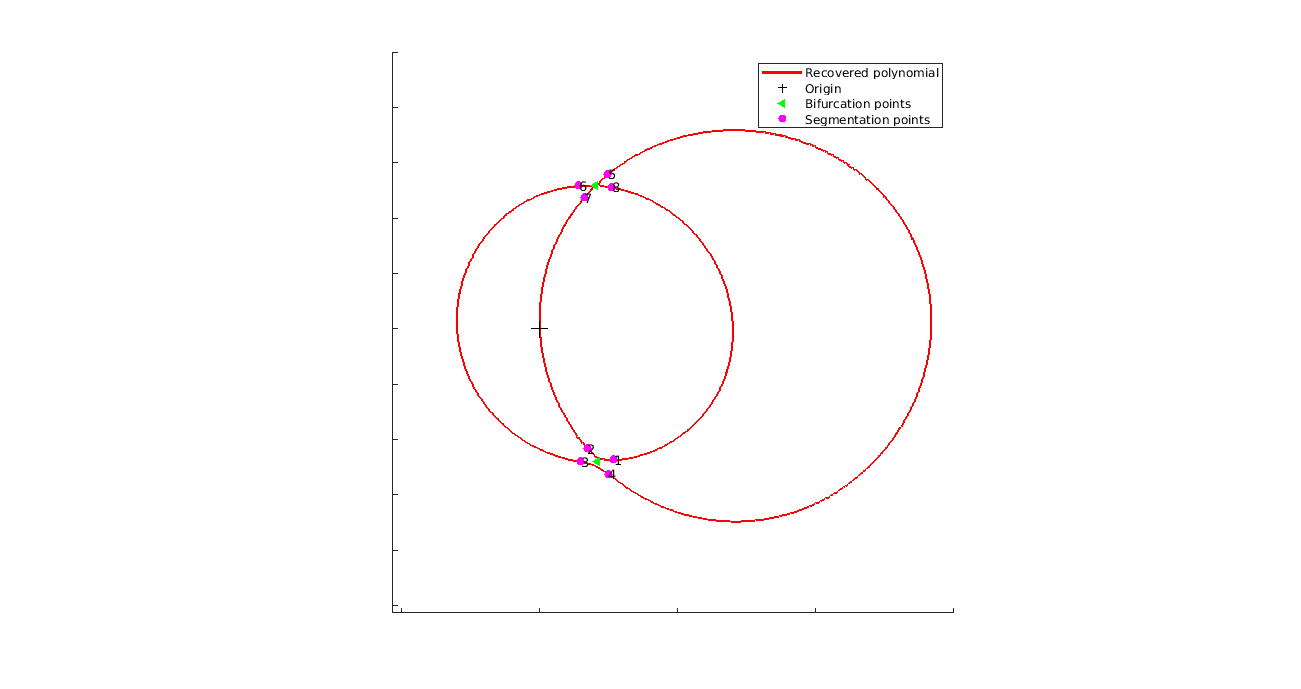}
  \caption{The level set of two overlapping circles gives rise to six distinct domains.}\label{Ex2a}
\end{figure}

\begin{figure}[H]
  \includegraphics[scale =0.6]{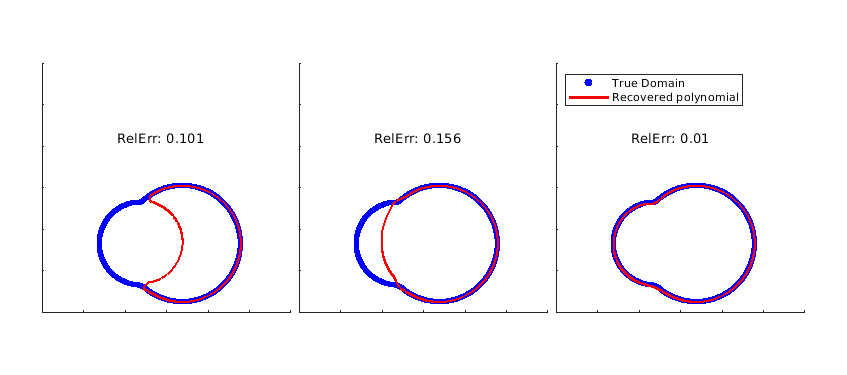}
  \caption{Conjoined circles.}\label{Ex2b}
\end{figure}

\begin{figure}[H]
  \includegraphics[scale =0.6]{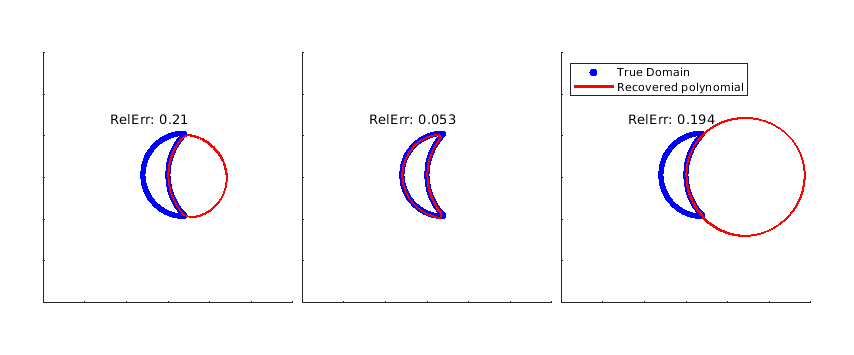}
  \caption{Crescent.}\label{Ex2c}
\end{figure}

\begin{figure}[H]
  \includegraphics[scale =0.6]{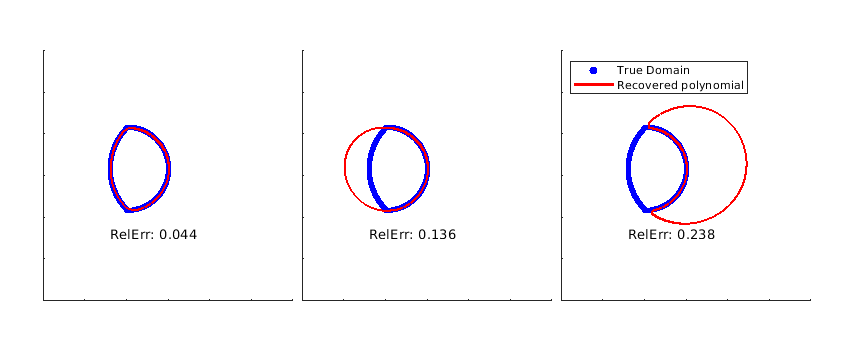}
  \caption{Intersection of circles.}\label{Ex2d}
\end{figure}

In the third example, we present in Figure \ref{Ex3} a  square with sinusoidal sides and its recovery from a single domain candidate.

\begin{figure}[H]
  \includegraphics[scale =0.4]{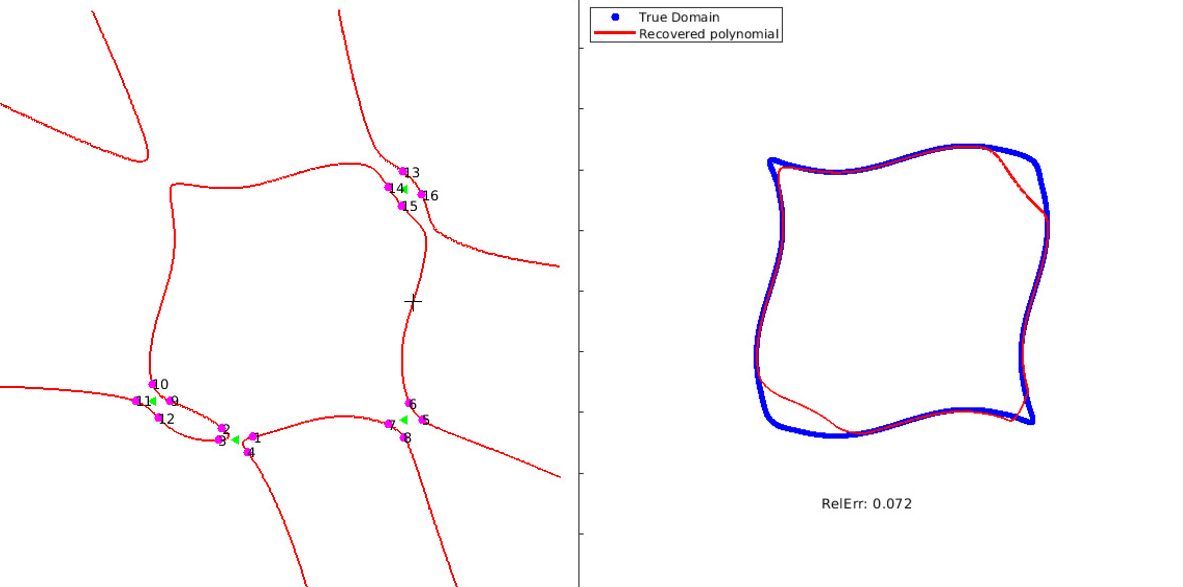}
  \caption{Domain recovery with a single candidate.}\label{Ex3}
\end{figure}

Finally, we show in  Figure \ref{Ex4}  that sometimes the recovered polynomial simply does not give the right domain. The true domain is in blue while the level set of the reconstructed polynomial from the GPTs is in red. This failure to recover the level set could stem from several reasons. The first reason is that higher degree domains are more unstable due to the higher powers taken in computing their GPTs. The second reason is that the proximity of the origin to a bifurcation point could cause instability. This however is still under investigation. We invite the reader to play around with the algorithm which is open source and available at https://github.com/JAndriesJ/ASPT.

\begin{figure}[H]
  \includegraphics[scale =0.6,center]{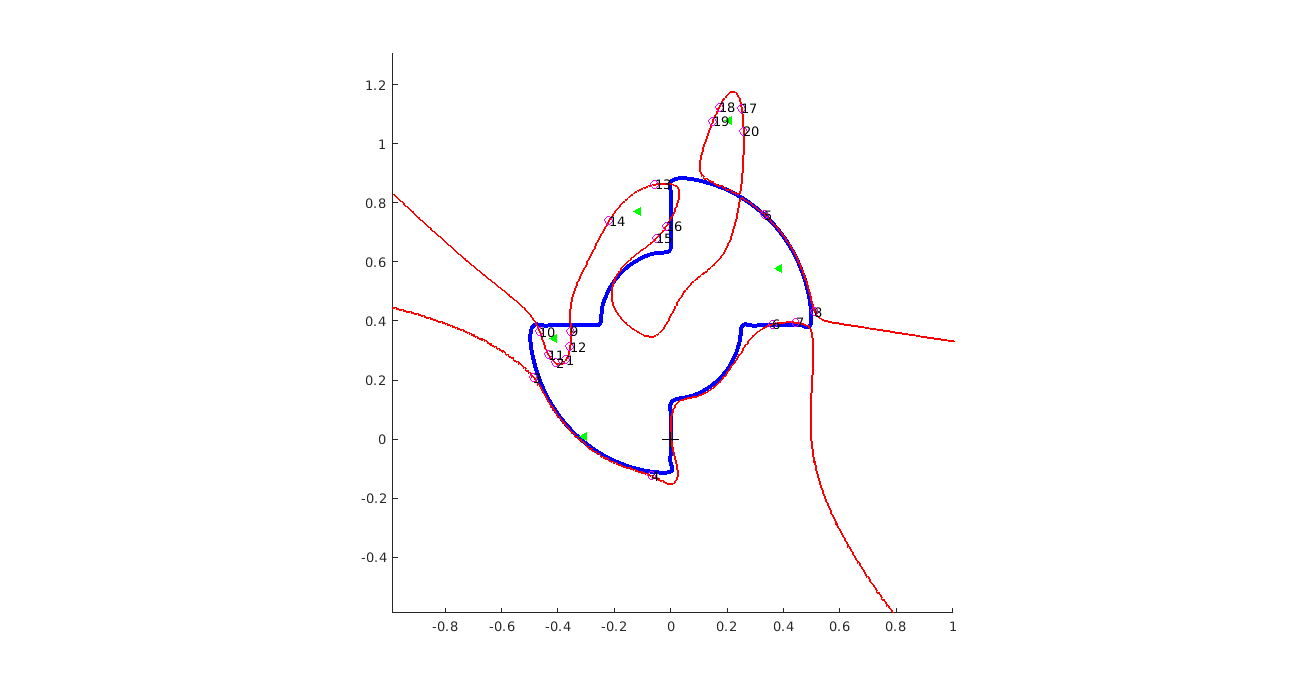}
  \caption{Failed polynomial recovery.}\label{Ex4}
\end{figure}

\end{document}